\documentclass[a4paper,11pt]{amsart}

\usepackage[utf8]{inputenc}
\usepackage[T1]{fontenc}
\usepackage{lmodern}
\usepackage{microtype}
\usepackage{hyperref}
\usepackage[top=3cm, bottom=3cm]{geometry}

\usepackage{enumerate}

\usepackage[abbrev]{amsrefs}

\newtheorem{proposition}{Proposition}[section]
\newtheorem{theorem}{Theorem}
\newtheorem{lemma}[proposition]{Lemma}

\theoremstyle{definition}
\newtheorem{definition}{Definition}[section]

\theoremstyle{remark}
\newtheorem{remark}{Remark}[section]

\numberwithin{equation}{section}

\usepackage{amssymb}

\newcommand{\abs}[1]{\lvert#1\rvert}

\newcommand{\st}{\: :\:}

\newcommand{\R}{\mathbb{R}}

\newcommand{\N}{\mathbb{N}}
\newcommand{\Z}{\mathbb{Z}}

\DeclareMathOperator{\capa}{cap}

\title[Poly\'a--Szeg\H o and relative capacity inequalities]{Equivalence between Poly\'a--Szeg\H o and relative capacity inequalities under rearrangement}
\author{Jean Van Schaftingen}

\address{Universit\'e catholique de Louvain\\
Institut de Recherche en Math\'ematique et Physique (IRMP)\\
Chemin du Cyclotron 2 bte L7.01.01\\
1348 Louvain-la-Neuve\\
Belgium}
\email{Jean.VanSchaftingen@uclouvain.be}


\begin{document}

\begin{abstract}
The transformations of functions acting on sublevel sets that satisfy a P\'olya--Szeg\H o inequality are characterized as those being induced by transformations of sets that do not increase the associated capacity.
\end{abstract}

\maketitle


\section{Introduction}

Symmetrization by rearrangement is an analytical tool to prove that some functionals achieve their minimum on symmetric sets and functions.
In the original setting \citelist{\cite{PolyaSzego1951}\cite{PolyaSzego1945}}, it associates to every nonnegative measurable function \(u : \R^n \to \Bar{\R}= \R \cup \{+ \infty\}\) a symmetrized function \(u^* : \R^n \to \Bar{\R}\), which is radial and such that for every \(t \in \R\),
\[
 \mathcal{L}^n \bigl(\bigl\{ x \in \R^n \st u^* (x) >t \bigr\}\bigr)
 =\mathcal{L}^n \bigl(\bigl\{ x \in \R^n \st u (x) >t \bigr\}\bigr).
\]

Two key properties of symmetrization in variational problems are the \emph{Cavalieri principle} \citelist{\cite{LiebLoss2001}*{\S 3.3 (3)}}: for every Borel measurable function \(f : [0, \infty) \to [0, \infty)\) such that \(f (0) = 0\) and \(f \circ u \in L^1 (\R^n)\), one has \(f \circ u^* \in L^1 (\R^n)\) and 
\begin{equation}
\label{eqCavalieri}
  \int_{\R^n} f \circ u^* = \int_{\R^n} f \circ u
\end{equation}
and the \emph{P\'olya--Szeg\H o inequality}\citelist{\cite{PolyaSzego1945}\cite{Polya1950}\cite{PolyaSzego1951}\cite{LiebLoss2001}*{lemma 7.17}\cite{Willem2003}*{theorem 8.3.14}}: if \(u\) belongs to the Sobolev space \(W^{1, 2} (\R^n)\) of weakly differentiable square integrable functions whose weak derivative is square integrable, then \(u^* \in W^{1, 2} (\R^n)\) also, and the weak derivatives \(du\) and \(du^*\) satisfy the inequality
\begin{equation}
\label{eqPolyaSzego}
 \int_{\R^n} \abs{du^*}^2 \le \int_{\R^n} \abs{du}.
\end{equation}

Many different types of symmetrizations by rearrangement can be defined depending on the particular problem \citelist{\cite{Kawohl1985}*{II.1}}. 
In this variety, it can be observed that some properties of the symmetrization only depend on the way that it acts on sublevel sets. For instance if \(\cdot^*\) is a transformation of measurable subsets of \(M\) into subsets of \(N\) that induces a transformations of functions in such a way that for every \(t \in \R\), 
\[
 \{y \in N \st u^* (y) > t \} = \{x \in M \st u (x) > t \}^*,                                                                                                                                                                                                                                                                                                                                                                                                                                                                                                                                                                                              
\]
then the Cavalieri principle \eqref{eqCavalieri} is equivalent with the preservation of measure: for every measurable set \(A \subseteq \R^n\), \(\mathcal{L}^n (A^*) = \mathcal{L}^n (A)\). Moreover, this implies that the transformation \(\cdot^*\) is nonexpansive for the Nikodym distance between measurable sets and from \(L^p (M)\) to \(L^p (N)\), for every \(p \in [1, \infty]\) \citelist{\cite{CroweZweibelRosenbloom1986}\cite{Willem2003}*{corollaire 21.10}}.

The characterization of the P\'olya--Szeg\H o inequality by the properties of the symmetrization by sets is more delicate. An analysis of the proof of this inequality by the coarea formula \cite{Talenti1976}*{proof of lemma 1} highlights as a tool an isoperimetric inequality on the sublevel sets: the De Giorgi perimeters \(P (A) \) and \(P (A^*)\) satisfy
\[
  P (A^*) \le P (A).
\]
This isoperimetric condition is sufficient in the counterpart of \eqref{eqPolyaSzego} for the total variation \(\int_{\R^n} \abs{du}\) when \(u\) is in the space of functions of bounded variation \(BV (\R^n)\).
However, it is not sufficient for the inequality \eqref{eqPolyaSzego}. 
Indeed, the Dirichlet integral \(\int_{\R^n} \abs{du}^2\) depends not only on the geometry of the sublevel sets of \(u\), but also on the relative position of its sublevel sets.

The goal of this note is to characterize by their action on sublevel sets rearrangements for which the P\'olya--Szeg\H o inequality holds.

\begin{theorem}
\label{theoremMain}
Let \(\cdot^*\) be a measure-preserving rearrangement.
One has for for every measurable sets \(A \subseteq B \subset \R^n\)
\[
 \capa (A^*, B^*) \le \capa (A, B),
 \]
if, and only if, for every function \(u \in W^{1, 2} (\R^n)\), one has
\(u^* \in W^{1, 2} (\R^n)\) and 
\[
 \int_{\R^n} \abs{d u^*}^2 \le \int_{\R^n} \abs{d u}^2.
\]
\end{theorem}

Here the relative capacity of a condenser is defined by 
\[
  \capa (A, B) = \inf \Bigl\{ \int_{\R^n} \abs{du}^2 : u \in W^{1, 1}_{\mathrm{loc}} (\R^n), \text{\(u = 0\) on \(\R^n \setminus B\) and \(u = 1\) on \(A\)}\Bigr\}. 
\]
This notion is different from usual definitions of relative capacity, but it coincides if \(A\) and \(B\) are bounded open sets
.

The necessity of the condition on capacities is already well-known. 
Indeed, symmetrization by rearrangement has been a fundamental tool to prove isoperimetric inequalities for capacities of condensers \citelist{\cite{PolyaSzego1945}\cite{Dubinin1985}\cite{Dubinin1987}\cite{Dubinin1991}\cite{Polya1950}\cite{PolyaSzego1951}\cite{Gehring1961}\cite{Sarvas1972}\cite{Wolontis1952}}. The essential new content in theorem~\ref{theoremMain} is that any transformation of sets that does not increase the capacity of condensers induces a transformation of functions that does not increase the Dirichlet integral. 

Theorem~\ref{theoremMain} is not a practical tool for proving P\'olya--Szeg\H o inequalities: capacity inequalities are not simpler conceptually to prove than P\'olya--Szeg\H o inequalities. Theorem~\ref{theoremMain} however asserts that P\'olya--Szeg\H o and capacity inequalities cannot be separated from each other.
Theorem~\ref{theoremMain} can be seen as a counterpart for the P\'olya--Szeg\H o inequality of the characterization of V. Maz\cprime{}ya of the Sobolev inequalities by inequalities involving the measure and the capacity of sets \citelist{\cite{Mazya1972}\cite{Mazya2005}\cite{Mazya2011}*{chapters 3 and 4}}.

Theorem~\ref{theoremMain} follows from a general equivalence between integral inequalities with weak derivatives and associated capacities for transformations acting on sublevel sets.
It is interesting that the preservation of the measure does not play in fact an essential role and that the principle extends to integrands that ensure local uniform integrability of the weak derivatives.

\section{Transformation of functions acting on level sets}

In this work, we shall consider transformations of functions acting on level sets (see \cite{VanSchaftingen2002}). 

\begin{definition}
The transformation \(\cdot^*\) is a \emph{transformation of functions acting on level sets} if every Borel measurable set \(A \subseteq M\) is mapped by \(\cdot^*\) to a Borel measurable set \(A^* \subseteq N\), if \(\emptyset^* = \emptyset\) and \(M^* = N\), and if every Borel measurable function \(u : M \to \Bar{\R}\) is mapped to a Borel measurable function \(u^* : N \to \Bar{\R}\) in such a way that for every \(t \in \R\),
\[
 \{y \in N \st u^* (y) > t \} = \{x \in M \st u (x) > t \}^*. 
\]
\end{definition}

Such transformations can be characterized by their action on sets.

\begin{proposition}
\label{propositionCharactSets}
A transformation of sets \(\cdot^*\) induces a transformations of functions acting on level 
sets if and only if 
\begin{enumerate}[(a)]
  \item \(\emptyset^* = \emptyset\) and \(M^* = N\),
  \item \(\cdot^*\) is monotone: if \(A \subseteq B\) are Borel measurable sets, then \(A^* \subseteq B^*\), 
  \item \(\cdot^*\) is continuous from the inside: if \((A_n)_{n \in \N}\) is an increasing sequence of Borel measurable subsets of \(M\), then 
\[
 \Bigl(\bigcup_{n \in \N} A_n\Bigr)^* = \bigcup_{n \in \N} A_n{}^*.
\]
\end{enumerate}
\end{proposition}

The presentation of the present work is thus consistent with axiomatic definitions of rearrangement \citelist{\cite{BrockSolynin2000}*{\S 3}\cite{Kawohl1985}*{\S II.2}\cite{Willem2003}*{\S 21}\cite{Sarvas1972}\cite{CroweZweibelRosenbloom1986}\cite{VanSchaftingenWillem2004}}, except that here we do not assume any measure-preserving properties.
This prevents us from negliging null sets for a measure (this is the case for different reasons in \cite{VanSchaftingen2006AIHP}).
\begin{proof}%
[Proof of proposition~\ref{propositionCharactSets}]
It can be checked that the function \(u^* : N \to \Bar{\R}\) defined for every \(y \in N\) by
\[
  u^* (y) = \sup\, \bigl\{t \in \R \st y \in \{x \in M \st u (x) > t \}^*\bigr\}
\]
satisfies the required property if \(\cdot^*\) is monotone and continuous from the inside.
\end{proof}

The transformations can also be characterized by their action on functions.

\begin{proposition}
Let \(\cdot^*\) be a transformation of Borel measurable functions.
Then, the transformation \(\cdot^*\) is induced by the action on level sets of a transformation of measurable functions if and only if 
\begin{enumerate}[(a)]
  \item \(0^* = 0\),
  \item \(\cdot^*\) is monotone: if \(u : M \to \Bar{\R}\) and \(v : M \to \Bar{\R}\) are Borel measurable functions and \(u \le v\) in \(M\), then \(u^* \le v^*\),
  \item \(\cdot^*\) is continuous: for every nondecreasing sequence of Borel measurable functions \((u_n)_{n \in \N}\) from \(M\) to \(N\), the sequence \((u_n{}^*)_{n \in \N}\)converges to the function \((\lim_{n \to \infty} u_n)^*\),
  \item \(\cdot^*\) commutes with composition: for every nondecreasing continuous function  \(f : \Bar{\R} \to \Bar{\R}\) and for every Borel measurable function \(u : M \to \R\), 
\[
 f \circ (u^*)= (f \circ u)^*.
\]
\end{enumerate}
\end{proposition}
\begin{proof}
First we choose a continuous nondecreasing function \(f : \Bar{\R} \to \Bar{\R}\) such that \(f (0) = 0\), \(f (1) = 1\) and \(f (t) \ne t\) if \(t \in \Bar{\R} \setminus \{0, 1\}\).
If \(A \subset M\) is a Borel measurable set, we have \(f \circ \chi_A= \chi_A\) and thus
\(
  (\chi_A)^*=(f \circ \chi_A)^* = f \circ \chi_A^*.
\)
Hence, \((\chi_A)^*\) is a characteristic functions and the transformation of sets can be defined by 
\(
 \chi_{A^*} = \bigl(\chi_A\bigr)^*.
\)
Given a real number \(t \in \R\) and a Borel measurable \(u : M \to \Bar{\R}\), we define a nondecreasing sequence of nondecreasing continuous maps \((f_n)_{n \in \N}\) from \(\Bar{\R}\) to \(\Bar{\R}\) that converges to \(\chi_{(t,\infty]}\). The nondecreasing sequence  
\((f_n \circ u)_{n \in \N}\) converges to \(\chi_{A_t}\),
where \(A_t=\{x \in M \st u (x) > t \}\). Therefore, by our assumption,
\[
  f_n \circ (u^*) = (f_n \circ u)^* \to \chi_{A_t^*},
\]
as \(n \to \infty\),
from which we conclude that \(\{y \in N \st u^* (y) > t\} = \{x \in M \st u (x) > t\}^*\).
\end{proof}

\section{The general equivalence}
Given an \(m\)--dimensional manifold \(M\), the nonnegative density bundle \(\mathcal{D}_+ M\) is a fiber bundle whose fibers are maps \(\delta : \Lambda^m T_x M \to [0, \infty)\) such that for every linear map \(A : T_x M \to T_x M\), \(\delta \circ \Lambda^m A = \abs{\det A} \delta\) (see for example \citelist{\cite{Lee2013}*{p. 427--432}\cite{BergerGostiaux1988}*{\S 3.3}}).
Given a nonnegative Borel measurable section \(\varphi : T^*M \to \mathcal{D}_+ M\) from the cotangent bundle to the density bundle and \(\lambda > 0\), we define the capacity of a pair of Borel-measurable sets \(A \subseteq B \subseteq M\) by
\begin{equation}
\label{eqCapa}
 \capa_{\varphi, \lambda} (A, B) = \inf \Bigl\{ \int_{M} \varphi (du) \st u \in W^{1, 1}_{\mathrm{loc}} (M) \text{ , } u = \lambda \text{ on \(A\) and } u = 0 \text{ in \(M \setminus B\)} \Bigr\}.
\end{equation}
In this definition, \(W^{1, 1}_\mathrm{loc} (M)\) denotes the set of functions \(u : M \to \R\) such that for every local chart \(\gamma : U \to M\), \(u \circ \gamma \in W^{1, 1}_\mathrm{loc} (U)\); this notion only depends on the differential structure of \(M\) and does not depend on a Riemannian structure on \(M\).

In this definition of capacity, the set of the infimum might be empty and the capacity infinite. 

Finally, we say that the bundle map \(\psi : T^*N \to \mathcal{D}_+ N \to \Bar{\R}^+\) is locally coercive if for every 
for every compact set \(K \subseteq N\)
\[
 \lim_{r \to \infty} \inf_{\substack{\xi \in T^*N\\ \pi (\xi) \in K\\ \abs{\xi} \ge r}} \frac{\abs{\psi (\xi)}}{\abs{\xi}} = \infty,
\]
where \(\abs{\cdot}\) denotes any continuous norms on the bundles \(T^*N\) and \(\mathcal{D}_+ N\), and \(\pi : T^*N \to N\) is the canonical projection of the bundle.

\begin{proposition}[Sufficient condition for the P\'olya--Szeg\H o inequality]
\label{propositionPolyaSzego}
Let \(\varphi : T^*M \to \mathcal{D}_+ M\) and  \(\psi : T^*N \to \mathcal{D}_+ N\) be Borel measurable bundle maps such that \(\varphi (0) = 0\) and \(\psi (0) = 0\).
Assume that \(\psi\) is locally coercive and for every \(y \in N\), \(\psi \vert_{T^*_y N}\) is a convex lower semi-continuous function.
If for every Borel-measurable sets \(A \subseteq B \subseteq M\) and every \(\lambda > 0\),
\[
 \capa_{\psi, \lambda} (A^*, B^*) \le \capa_{\varphi, \lambda} (A, B),
\]
if \(u \in W^{1, 1}_{\mathrm{loc}} (M)\) is a Borel measurable function, if \(u^* \in L^1_\mathrm{loc} (N)\) and if 
\[
  \int_{M} \varphi (du) < \infty,
\]
then \(u^* \in W^{1, 1}_{\mathrm{loc}} (M)\) and  
\[
\int_{N} \psi (du^*) \le \int_{M} \varphi (du).
\]
\end{proposition}

Proposition~\ref{propositionPolyaSzego} is stated in a quite general framework. It works in particular for rearrangement with respect to an arbitrary convex set \citelist{\cite{AlvinoFeroneLionsTrombetti1997}\cite{VanSchaftingen2006AIHP}\cite{Cianchi2007}}, for monotone rearrangements on cylinders, acting either fiberwise \citelist{\cite{BerestyckiLachandRobert2004}\cite{VanSchaftingen2006PAMS}} or globally on the cylinder \citelist{\cite{Alberti2000}\cite{Carbou1995}}.
It also applies on spheres or the hyperbolic spaces \cite{Baernstein1994}, weighted inequalities \citelist{\cite{GarsiaRodemich}\cite{Landes2007}}, and to inequalities in which the functional is also symmetrized \citelist{\cite{Klimov1999}\cite{VanSchaftingen2006AIHP}\cite{Cianchi2007}}.

The local coercivity assumption is essential. Indeed, if \(M = N = \R\),
\[
A^* = 
\begin{cases}
(0, \mathcal{L}^1 (A)) & \text{if \(\mathcal{L}^1 (A) \le 1\)},\\
(\mathcal{L}^1(A)/2 - \frac{1}{2}, \mathcal{L}^1 (A) + \tfrac{1}{2}) & \text{if \(\mathcal{L}^1 (A) > 1\)},
\end{cases}
\]
\(\varphi (\xi) = \psi (\xi) = \abs{\xi}\) and \(u (x) = \max (1 - \abs{x}, 0)\), then \(u \in W^{1, 1}_{\mathrm{loc}} (\R)\), \(\int_{\R} \varphi (du) < \infty\) and \(u^* \not \in W^{1, 1}_\mathrm{loc} (\R)\).
The reader will observe that \(u^*\) is still a function of bounded variation; we shall not pursue in this direction.

Our proof of proposition~\ref{propositionPolyaSzego} will rely on the following lower semi-continuity result.

\begin{lemma}
\label{lemmaCoercivityCompactness}
Let \(\psi : T^*N \to \Bar{\R}^+\) be a Borel measurable bundle map such that \(\psi (0) = 0\).
Assume that \(\psi\) is locally coercive and for every \(y \in N\), \(\psi \vert_{T^*_y N}\) is a convex lower semi-continuous function.
If the sequence \((v_n)_{n \in \N}\) converges to \(v : N \to \Bar{\R}\) in \(L^1_{\mathrm{loc}} (N)\) and if 
\[
 \liminf_{n \to \infty} \int_{N} \psi (d v_n) < \infty,
\]
then \(v \in W^{1, 1}_\mathrm{loc} (N)\) and 
\[
 \int_{N} \psi (d v) \le \liminf_{n \to \infty} \int_{N} \psi (d v_n).
\]
\end{lemma}

\begin{proof}
Without loss of generality, we can assume that 
\[
  \limsup_{n \to \infty} \int_{N} \psi (d v_n)=\liminf_{n \to \infty} \int_{N} \psi (d v_n).
\]
Given \(\varepsilon > 0\), by assumption there exists \(r > 0\) such that if \(\xi \in T^*N\), \(\pi (\xi) \in K\) and \(\abs{\xi} \ge r\), then
\[
 \abs{\xi} \le \varepsilon \abs{\psi (\xi)}.
\]
Since \(\psi\) is nonnegative, for every \(\xi \in \pi^{-1} (K)\), 
\[
 \abs{\xi} \le r + \varepsilon \abs{\psi (\xi)}.
\]
Since the norm on the densities is continuous, there exists a continuous density \(w : M \to \mathcal{D}_+ M\) such that for every \(\delta \in \mathcal{D}_+ M\), \(\abs{\delta}w = \delta\).
Then, for every \(n \in \N\)  and every set \(A \subseteq N\)
\[
 \int_{A} \abs{d v_n}w \le r \int_{A} w + \varepsilon \int_{A} \psi (d v_n).
\]
Hence if the set \(A\) is taken so that \(\int_{A} w \le \frac{\varepsilon}{r}\),
\[
 \int_{A} \abs{d v_n}w \le \varepsilon \Bigl( 1 + \int_{N} \psi (dv_n) \Bigr).
\]
By the boundedness assumption, we have proved that the sequence \((\abs{d v_n})_{n \in \N}\) is equi-integrable on every compact compact subset \(K \subset N\). By the Dunford--Pettis compactness criterion \citelist{\cite{Wojtaszczyk1991}*{theorem III.C.17}\cite{DunfordSchwartz}*{theorem IV.8.9}\cite{Bogachev2007}*{theorem 4.7.18}}, the sequence \((d v_n)_{n \in \N}\) is weakly compact in \(L^1(K)\). Since the sequence \((v_n)_{n \in \N}\) converges to \(v\) in \(L^1 (K)\), it converges to \(v\) in \(W^{1, 1}_\mathrm{loc} (N)\).
Recalling that for every \(y \in N\), the function \(\psi \vert_{T_y^* N}\) is nonnegative, lower semicontinuous and convex, we conclude that for every compact set \(K \subseteq N\),
\[
  \int_{K} \psi (dv) \le \lim_{n \to \infty} \int_{N} \psi (d v_n).
\]
(see for example \cite{EkelandTemam1976}*{corollary 2.2}), and we conclude with Lebesgue's monotone convergence theorem, by writing the manifold \(N\) as a countable union of compact subsets.
\end{proof}

\begin{proof}[Proof of proposition~\ref{propositionPolyaSzego}]
We define for every \(t \in \R\) the sublevel set 
\[
  A_t = \{ x \in \R^n : u (x) > t \}.
\]
For every \(n \in \N_*\) and for every \(k \in \Z\), by our definition of capacity \eqref{eqCapa}, there exists a function \(v_{n, k} \in W^{1, 1}_\mathrm{loc} (N)\) such that \(0 \le v_{n, k} \le \frac{1}{n}\), \(v_{n, k} = 1/n\) on \(A_{k/n}^*\), \(v_{n, k} = 0\) on \(N \setminus A_{(k-1)/n}^*\), and 
\[
  \int_{N} \psi (d v_{n, k}) \le 
   \capa_{\psi,1/n} (A_{k/n}^*, A_{(k-1)/n}^*) + \frac{1}{n^3}.
\]
By our the capacity inequality of our assumption and the definition of capacity \eqref{eqCapa} again, we deduce that 
\[
  \int_{N} \psi (d v_{n, k})    
   \le  \capa_{\varphi,1/n} (A_{k/n}, A_{(k-1)/n}) + \frac{1}{n^3}
   \le \int_{A_{k/n} \setminus A_{(k-1)/n}} \varphi (d u) + \frac{1}{n^3}.
\]
We define the function \(v_n : N \to \R\) by 
\[
  v_n = - n + \sum_{k=-(n^2 - 1)}^{n^2} v_{n,k} 
\]
One has, since \(\varphi\) is nonnegative and \(\varphi (0) = 0\),
\[
 \int_{N} \psi (dv_n) \le \sum_{k=-(n^2 - 1)}^{n^2} \Bigl(\capa_{\varphi,1/n} (A_{k/n}, A_{(k-1)/n}) + \frac{1}{n^3}\Bigr)
 \le \int_{M} \varphi (d u) + \frac{2}{n}.
\]
Since \(\abs{v_n - u^*} \le \frac{1}{n}\) on \(A_n\setminus A_{-n}\), \(\abs{v_n} \le \abs{u_*} + \frac{1}{n}\) in \(N\) and \(u_* \in L^1_{\mathrm{loc}} (N)\), by Lebesgue's dominated convergence theorem, the sequence \((v_n)_{n \in \N}\) converges to \(u^*\) in \(L^1_{\mathrm{loc}} (N)\).
In view of lemma~\ref{lemmaCoercivityCompactness}, we have \(u^* \in W^{1, 1}_{\mathrm{loc}} (M)\) and 
\[
 \int_{N} \psi (d u^*) \le \liminf_{n \to \infty} \int_{N} \psi (d v_n)
 \le \int_{M} \varphi (du).\qedhere
\]
\end{proof}

Parts of this proof are reminiscent of the proof of the conductor inequality \cite{Mazya2011}*{proposition 3.3}.

\begin{remark}
The assumption that \(u^*\) is locally summable can be avoided by working with Sobolev spaces by truncation: a function \(u : M \to N\) is in \(\mathcal{T}^{1, 1}_{\mathrm{loc}} (M)\) if for every \(K \in [0, \infty)\), \(u_K = \max (\min (u, K), -K) \in W^{1, 1}_{\mathrm{loc}} (M)\); one defines then \(d u = d u_K\) on \(u^{-1}([-K, K])\) \cite{BBGGPV1995} (Such functions qeneralize to colocally weakly differentiable maps between manifolds \cite{ConventVanSchaftingen}.)
For every \(K \in [0, \infty)\), \(u_K \in L^\infty (M)\), and thus \(u_K^* \in L^\infty (M)\). If we assume that \(u \in \mathcal{T}^{1, 1}_\mathrm{loc} (M)\) and that
\[
\int_{M} \varphi (d u) < \infty  
\]
then for every \(K \in [0, \infty)\), one has \(u_K \in W^{1, 1}_\mathrm{loc} (M)\) and 
\[
\int_{M} \varphi (d u_K) \le \int_{M} \varphi (d u) < \infty   . 
\]
By proposition~\ref{propositionPolyaSzego}, \(u_K^* \in W^{1, 1}_\mathrm{loc} (N)\), and 
\[
  \int_{N} \varphi (d u_K^*) \le \int_{M} \varphi (d u_K) \le \int_{M} \varphi (d u).
\]
By definition, \(u^* \in \mathcal{T}^{1, 1}_\mathrm{loc} (N)\), and by Lebesgue's monotone convergence theorem
\[
  \int_{N} \varphi (d u^*) = \lim_{K \to \infty}  \int_{N} \varphi (d u_K^*) \le \int_{M} \varphi (d u).
\]
\end{remark}

We show that the capacity inequality of proposition~\ref{propositionPolyaSzego} is necessary to have a P\'olya--Szeg\H o inequality.

\begin{proposition}[Necessary condition for the P\'olya--Szeg\H o inequality]
\label{propositionPolyaSzegoNecessary}
Let \(\varphi : T^*M \to \Bar{\R}_+\) and \(\psi : T^* N \to \Bar{\R}_+\) be Borel measurable functions such that \(\varphi = 0\) and \(\psi (0) = 0\).
If for every bounded Borel measurable function \(u : M \to \R\) such that  \(u \in W^{1, 1}_{\mathrm{loc}} (M)\) and 
\[
  \int_{M} \varphi (du) < \infty,
\]
one has \(u^* \in W^{1, 1}_{\mathrm{loc}} (N)\) and 
\[
\int_{N} \psi (du^*) \le \int_{M} \varphi (du),
\]
then for every Borel-measurable sets \(A \subset B \subset N\) and every \(\lambda > 0\),
\[
 \capa_{\psi, \lambda} (A^*, B^*) \le \capa_{\varphi, \lambda} (A, B).
\]
\end{proposition}
\begin{proof}
Let \(u \in W^{1, 1}_{\mathrm{loc}} (M)\) be Borel-measurable, such that \(u = 0\) on \(M \setminus B\) and \(u = \lambda\) on \(A\). If 
\[
  \int_{M} \varphi (du) < \infty, 
\]
we can assume by truncation that \(0 \le u \le \lambda\). By assumption, we deduce that \(u^* \in W^{1, 1}_{\mathrm{loc}} (N)\) and 
\[
 \int_{M} \varphi (du^*) \le \int_{M} \varphi (du).
\]
In order to conclude that \(\capa_{\varphi,\lambda} (A^*, B^*) \le \int_{M} \varphi (du)\), we need to show that \(u = 0\) on \(N \setminus B\) and that \(u = \lambda\) on \(A\).
We first observe that
\[
 B \supseteq \{x \in M \st u (x) > 0\}
\]
so that by monotonicity of \(\cdot^*\),
\[
 B^* \supseteq \{x \in M \st u (x) > 0\}^*=\{y \in N \st u^* (y) > 0\},
\]
that is, \(u = 0\) on \(N \setminus B^*\).
Next, for every \(t \in [0, 1)\),
\[
 A \subseteq \{x \in M \st u (x) > t\},
\]
hence 
\[
 A^* \subseteq \{x \in M \st u (x) > t\}^*=\{y \in N \st u^* (y) > t\}.
\]
Therefore,
\[
 A^* \subseteq \bigcap_{t \in [0, \lambda)} \{y \in N \st u^* (y) > t\} = \{y \in \R^n \st u^* (y) \ge \lambda\}.\qedhere
\]
\end{proof}

We deduce now theorem~\ref{theoremMain} as a particular case of propositions~\ref{propositionPolyaSzego} and \ref{propositionPolyaSzegoNecessary}.

\begin{proof}[Proof of theorem~\ref{theoremMain}]
In order to prove the P\'olya--Szeg\H o inequality, we first note that if \(\varphi : \R^n \times \R^n \to \mathcal{D}_+\R\) is defined for every \(x \in \R^n\), \(\xi \in \R^n\) by \(\varphi (x,\xi) = \abs{\xi}^2 \abs{\det}\), then for every Borel measurable sets \(A \subseteq B \subseteq \R^n\),
\[
  \capa_{\varphi, \lambda} (A, B) = \lambda^2 \capa (A, B).
\]
If \(u \in W^{1, 2} (\R^n)\), then, by definition, \(u \in L^2 (\R^n)\) and by the Cavalieri principle \eqref{eqCavalieri}, \(u^* \in L^2 (\R^n) \subset L^1_\mathrm{loc} (\R^n)\) and thus by proposition~\ref{propositionPolyaSzego}, \(u^* \in W^{1, 1}_\mathrm{loc} (\R^n)\) and  
\[
  \int_{\R^n} \abs{d u^*}^2 \le \int_{\R^n} \abs{du}^2\;.
\]

For the converse statement, we know that for every \(u \in L^2 (\R^n) \cap W^{1, 1}_{\mathrm{loc}} (\R^n)\), if \(du \in L^2 (\R^n)\), then \(u^* \in W^{1, 1}_{\mathrm{loc}} (\R^n)\) and 
\[
 \int_{\R^n} \abs{du^*}^2 \le \int_{\R^n} \abs{du}^2.
\]
In order to apply proposition~\ref{propositionPolyaSzegoNecessary}, we need to extend this to the case where \(u \in L^{\infty} (\R^n) \cap W^{1, 1}_{\mathrm{loc}} (\R^n)\) and \(du \in L^2 (\R^n)\).
Without loss of generality, we assume that 
\begin{equation}
\label{eqFinite}
  \inf\,\bigl\{ t \in \R \st \mathcal{L}^n(\{x \in \R^n\st u (x) > t\}) < \infty\bigr\} = 0
\end{equation}
and we define \(u_n = (u - \frac{1}{n})_+\).
By \eqref{eqFinite}, \(u_n \in L^2 (\R^n)\) and thus by the hypothesis,
\[
 \int_{\R^n} \abs{du_n{}^*}^2 \le
 \int_{\R^n} \abs{du_n}^2 \le
 \int_{\R^n} \abs{du}^2.
\]
Since \(u_n^* = (u^* - \frac{1}{n})_+\), we conclude by Lebesgue's monotone convergence theorem that \(u^* \in W^{1, 1}_\mathrm{loc} (\R^n)\) and that 
\[
  \int_{\R^n} \abs{du^*}^2 \le \int_{\R^n} \abs{d u}^2. 
\]
We conclude by the necessary condition for a P\'olya--Szeg\H o inequality of proposition~\ref{propositionPolyaSzegoNecessary}.
\end{proof}

\end{document}